\documentclass[leqno]{amsart}
\usepackage[latin1]{inputenc}
\usepackage[T1]{fontenc}
\usepackage{amssymb, amsmath, color, textcomp, url,tikz, mathtools}

\usetikzlibrary{matrix,arrows}
\usetikzlibrary{fit}
\usetikzlibrary{positioning}
\usepackage[labelformat=empty]{caption}
\usepackage{mdwlist}
\usepackage{etoolbox}

\RequirePackage{ifpdf}
\ifpdf
   \usepackage[pdftex]{hyperref}
\else
   \usepackage[hypertex]{hyperref}
\fi

\usepackage{enumerate,xspace}

\theoremstyle{plain}
\newtheorem{theorem}{Theorem}[section]

\newtheorem{prop}[theorem]{Proposition}
\newtheorem{lemma}[theorem]{Lemma}

\newcounter{claimCount}
\AtBeginEnvironment{proof}{\setcounter{claimCount}{0}}

\theoremstyle{definition}
\newtheorem{remark}[theorem]{Remark}

\newtheorem{definition}[theorem]{Definition}
\newtheorem{example}[theorem]{Example}
\newtheorem*{question}{Question}

\newcommand{\nc}{\newcommand}

\nc{\Z}{\mathbb{Z}}
\nc{\Q}{\mathbb{Q}}
\nc{\N}{\mathbb{N}}
\nc{\F}{\mathbb{F}}
\nc{\UU}{\mathbb{U}}
\nc{\C}{\mathbb{C}}

\nc{\M}{\mathcal{M}}
\nc\LL{\mathcal L}
\nc\II{\mathcal I}

\nc{\stt}{\operatorname{St}}
\nc{\stab}{\operatorname{Stab}}
\nc{\GO}[1]{G_{#1}^{00}}
\nc{\Cos}[1]{\operatorname{Cos}(#1)}

\nc{\band}[1]{\bar d_{\mathcal{#1}}}
\nc\BD{\operatorname{BD}}
\nc{\bound}{\operatorname{Bound}}

\nc{\dcl}{\operatorname{dcl}}
\nc{\dclq}{\operatorname{acl^\text{eq}}}

\nc{\acl}{\operatorname{acl}}
\nc{\aclq}{\operatorname{acl^\text{eq}}}
\nc{\nf}[1]{_{\mid {#1}}}
\nc{\restr}[1]{\xspace_{\upharpoonright {#1}}}
\nc{\sbgp}[1]{\langle\xspace {#1}\xspace\rangle}

\nc{\str}[1]{\langle\xspace {#1}\xspace\rangle}
\nc{\strp}[1]{\langle\xspace {#1}\xspace\rangle_\text{pure}}

\nc\CAN{\operatorname{CB}}
\nc\inv{ ^{-1}}

\nc{\tp}{\operatorname{tp}}
\nc\cb{\operatorname{Cb}}
\nc\U{\operatorname{U}}
\nc{\cf}{\text{cf.\,}}
\nc{\eg}{\text{e.g. }}

\def\Ind#1#2{#1\setbox0=\hbox{$#1x$}\kern\wd0\hbox to
	0pt{\hss$#1\mid$\hss} \lower.9\ht0\hbox to
	0pt{\hss$#1\smile$\hss}\kern\wd0}
\def\Notind#1#2{#1\setbox0=\hbox{$#1x$}\kern\wd0\hbox to
	0pt{\mathchardef\nn="0236\hss$#1\nn$\kern1.4\wd0\hss}\hbox to
	0pt{\hss$#1\mid$\hss}\lower.9\ht0 \hbox to
	0pt{\hss$#1\smile$\hss}\kern\wd0}

\def\indip{\mathop{\ \ \hbox to 0pt{\hss$\mid^{\hbox to
				0pt{$\scriptstyle P$\hss}}$\hss}
		\lower4pt\hbox to 0pt{\hss$\smile$\hss}\ \ }}
\def\nindip{\mathop{\ \ \hbox to 0pt{\hss$\!\not{\mid}^{\hbox to
				0pt{$\scriptstyle\, P$\hss}}$\hss}
		\lower4pt\hbox to 0pt{\hss$\smile$\hss}\ \ }}

\begin{document}

\title[Supersimple expansions of the integers]{An exposition on the supersimplicity of certain expansions of the additive group of the integers}

\date{\today}

\author{Amador Martin-Pizarro and Daniel Palac\'in}
\address{Mathematisches Institut,
  Albert-Ludwigs-Universit\"at Freiburg, D-79104 Freiburg, Germany}
\email{pizarro@math.uni-freiburg.de}

\address{Departamento de \'Algebra, Geometr\'ia y Topolog\'ia, 	
	Facultad de Ciencias Matem\'aticas, 
		Universidad Complutense de Madrid, Plaza Ciencias 3, 28040, 
		Madrid, Spain}
	\email{dpalacin@ucm.es}
 
\thanks{The first author conducted reseaarch supported by the project BA6785/1-1 of the German Science Foundation (DFG). The second author conducted
research supported by project STRANO PID2021-122752NB-I00 and Grupos UCM 910444.}

\begin{abstract}
 In this short note, we present a self-contained exposition of the supersimplicity of certain expansions of the additive group of the integers, such as adding a generic predicate (due to Chatzidakis and Pillay), a predicate for the square-free integers (due to Bhardwaj and Tran) or a predicate for the prime integers (due to Kaplan and Shelah, assuming Dickson's conjecture). 
\end{abstract}

\maketitle


\section{Introduction}

Kaplan and Shelah \cite{KS17} showed that the theory of the additive group $\Z$ together with a unary predicate for the prime integers is supersimple of rank $1$, assuming Dickson's conjecture. A similar result was shown unconditionally by Bhardwaj and Tran \cite{BT21} for the theory of the additive group $\Z$ together with a unary predicate for the square-free integers. Whilst their proofs do not seem identical at first glance, they share a common structure. First, a partial quantifier-elimination result allows one to reduce the study of the theory to the study of certain formulae. Second, a \emph{randomness} condition on the predicate allows one to show that such formulae never divide unless they define a finite set.  

The goal of this note is to provide a general criteria (of a model-theoretic nature) which allow one to conclude that a certain expansion of (a sufficiently saturated elementary extension of) the additive group of the integers is supersimple of rank $1$. Our approach is also valid for the expansion given by a (symmetrized version of the) generic predicate \cite{CP98}. 

We would like to thank the anonymous referee for the detailed comments which have improved this exposition. 

\section{The setting}\label{S:Setting}

Throughout this note, we will denote by $\LL$ the language of (abelian) groups with a constant symbol for the element $1$ of $\Z$, as well as unary predicates $\equiv_k 1+\stackrel{j}{\ldots}+1$, for $0\le j< k$, which are interpreted as the elements in the group congruent to $j$ modulo $k$. The $\LL$-theory of the structure $\mathcal Z$ with universe $\Z$ is complete and has quantifier elimination \cite{wB76}. Moreover, working inside a sufficiently saturated expansion $\UU$ of $\mathcal Z$, the structure $\str{A}_{\LL}$ generated by a subset $A$ of $\UU$ is the subgroup $\sbgp{A, 1}$ generated by $A\cup\{ 1\}$ and the $\LL$-algebraic closure of $A$ is the pure closure $\strp A$ of elements $x$ of $\UU$ such that $mx=x+\stackrel{m}{\ldots}+x$ belongs to $\str{A,1}$. Clearly $\str{A}_{\LL}\subset \strp A$. In an abuse of notation, if $x\equiv_k j$, we set $\frac{x-j}{k}$ as the unique element $y$ with $ky+j=x$. 

Given now a new unary predicate $P$ (not in $\LL$), we consider two expansions of the language $\LL$, one given by  $\LL_P=\LL\cup\{P\}$ and another given by $\LL_P^+=\LL\cup\{P_n\}_{1\le n\in \N}$, where we interpret the predicate $P_n$ in $\Z$ as follows: \[ P_n(x) \text{ holds } \  \iff \ x\equiv_n 0 \land  
P(x/n). 
\]
In particular, in every elementary extension $\mathcal M$ of the $\LL_P^+$-structure with universe $\Z$, we have that  $P_n(\mathcal M)=nP(\mathcal M)$ for $1\le n$. Every predicate $P_n$ is definable in $\LL_P$ yet $P_n$ is not quantifier-free definable if $n\ge 2$. We will impose that $P$ (and thus every $P_n$) is \emph{symmetric}, that is,  
\[ \text{An element $a$ satisfies $P$ if and only if $-a$ does. }   \tag{P0}  \]

\begin{remark}\label{R:trivial_obs}
Given positive integers $m$, $n$ and $k$ as well as elements $a$ and $b$ of $\UU$, we have that 
\[ 
P_n( m a+ b) \text{ holds } \ \Longleftrightarrow \ P_{nk}( kma+ kb)  \text{ holds}. 
\] 
Therefore, whenever we consider the quantifier-free $\LL_P^+$-type of elements $m_1a+b_1, \ldots, m_r a+b_r$, we may always assume that $m_1=\ldots=m_r=m\ge 1$ (taking the least common multiple of the $m_i$'s). 
\end{remark}

A fundamental component in the proof of Kaplan and Shelah \cite{KS17} of the supersimplicity of the theory of $\Z$ with a distinguished unary predicate for the prime integers was Dickson's conjecture. This conjecture, which is still open and unlikely to be answered in a foreseable future, can be seen as a weakening of a condition already studied by Chatzidakis and Pillay for the theory of a geometric structure with a \emph{generic} predicate \cite{CP98}. 

\begin{definition}\label{D:add_gen}\textup{(}\cf \cite[Theorem 2.4]{CP98}\textup{)}
Consider an elementary extension $\mathcal M$ of $\mathcal Z$.  A subset $D$ of $M$ is \emph{additively random} if for every $r$ and $r'$ in $\N$ (where $r$ or $r'$ are possibly $0$) and terms $m_1 x+z_1, \ldots, m_rx+z_r$ as well as  $m'_1 x+z'_1, \ldots, m'_{r'}x+z'_{r'}$ with all $m_i,m'_j\ge 1$, there exists an $\LL$-formula $\chi(\bar z, \bar z')$ and a finite subset $\bound(\chi)$ of $\Z$ (\emph{the boundary} of $\chi$) such that the following conditions hold whenever $(m_i, b_i) \ne  (m'_j, b'_j)$ with $1\le i\le r$ and $1\le j\le r'$: 
\begin{enumerate}[(P1)]
    \item If $\chi(\bar b, \bar b')$ holds in $\mathcal M$, then there are infinitely many elements $a$'s in $M$ with 
    \[ 
    \bigwedge_{i=1}^r m_i a+b_i \in D \land  \bigwedge_{j=1}^{r'} m'_j a+b'_j \notin D.
    \] 
    \item If $\chi(\bar b, \bar b')$ does not hold in $\mathcal M$, then $r\ge 1$ and there exists some $k$ in the boundary $\bound(\chi)$ such that for every integers $s$ in $\{0,\ldots, k-1\}$, one of the terms $m_is+b_i$ is congruent with $0$ modulo $k$, that is,
    \[
    \mathcal M\models \forall \bar y \bar y' \left( \big(\neg\chi(\bar y,\bar y')\wedge \bigwedge_{i,j} \neg(y_i = y_j') \bigg) \to \bigvee_{k\in \bound(\chi)} \ \bigwedge_{s<k} \ \bigvee_{i=1}^{r} \ m_is+y_i\equiv_k 0 \right). 
    \]
    \item For every $k$ in $\bound(\chi)$, the solution set of $(x\in D) \land x\equiv_k 0$ in $\mathcal M$ is a finite subset of $\Z$. 
\end{enumerate}
\end{definition}
Observe that being additively random is an elementary property of the theory of $(\mathcal M, D)$ in the language $\LL_P$. 
\begin{remark}\label{R:add_gen}
With the notation of Definition \ref{D:add_gen}, whenever $(m_i, b_i) \ne  (m'_j, b'_j)$ with $1\le i\le r$ and $1\le j\le r'$, we have that (P1)-(P3) imply that $\chi(\bar b, \bar b')$ holds in $\mathcal M$ if and only if there are infinitely many elements $a$ in $M$ with \[ 
    \bigwedge_{i=1}^r m_i a+b_i \in D \land  \bigwedge_{j=1}^{r'} m'_j a+b'_j \notin D.
    \] 
Indeed, we need only show one implication. Assume therefore that there are infinitely many such elements $(a_n)_{n\in\N}$ in $M$ as above, yet $\chi(\bar b,\bar b')$ does not hold. In particular, we have that $r\ge 1$  by property (P2) and thus $m_ia_n+b_i$ belongs to $D$ for every $1\le i\le r$ and every $n$ in $\N$. Furthermore, property (P2) also yields the existence of some $k$ in $\bound(\chi)$ such that for infinitely many $n$'s we have $m_{i_0}a_n+b_{i_0}\equiv_k 0$ for some $1\le i_0\le r$, which clearly contradicts property (P3), since $m_{i_0}\ne 0$.  
\end{remark}
\begin{example}\label{E:add-random_ex}~

\begin{enumerate}[(a)]
\item The set of prime integers is additively random \cite{KS17} in $\mathcal Z$, under the assumption that Dickson's conjecture holds. Indeed, the associated formula $\chi(\bar z, \bar z')$ is given by their property $\star_{\bar f}$ in \cite[Lemma 2.3 \& Remark 2.5]{KS17} with $f_i=m_ix+z_i$ for $1\le i\le r$ and the boundary $\bound(\chi)$ consists of all the prime integers strictly less than $N=\max(m_1,\ldots,m_r, r)$+1 \cite[Remark 2.1]{KS17}. More precisely, 
\[
\chi(\bar z, \bar z')=\bigwedge_{p<N \text{ prime}} \ \bigvee_{s<p} \ \bigwedge_{i=1}^{r} \ m_is+z_i\not \equiv_p 0.
\] 
The condition (P3) holds, since an element $x$ which is divisible by $p$ and belongs to the prime integers is either $p$ or $-p$. 

\item The set of square-free integers is additively random \cite{BT21} in $\mathcal Z$. Indeed, the associated formula $\chi(\bar z, \bar z')$ is given  by 
\[ 
\bigwedge_{p \le  B \text{ prime}} \exists x \left(  \bigwedge_{i=1}^r m_ix+z_i\not\equiv_{p^2} 0 \right),
\] 
obtained by existential quantification from Remark \ref{R:trivial_obs} and the \emph{associated $p$-conditions} \cite[p. 1330, Lemmata 2.7 \& 2.11 \& Theorem 2.14]{BT21} where the boundary consists of all the integers $k=p^2$ with $p \le B= \max(\prod_i m_i, r)+1$ prime \cite[p. 1331]{BT21}.  Condition (P3) holds trivially, since a square-free integer cannot be congruent with $0$ modulo $p^2$. 
 \item The set $D=P^\mathcal{M}\cap -P^\mathcal{M}$ is additively random in every model $(\mathcal M,P^\mathcal M)$  of the theory $\mathrm{Th}(\mathcal Z)_P$ in the sense of Chatzidakis and Pillay \cite{CP98}. Indeed, the $\LL$-theory  $\mathrm{Th}(\mathcal Z)$ eliminates $\exists^\infty$ by quantifier-elimination \cite{wB76}, so the theory $\mathrm{Th}(\mathcal Z)_P$ of the generic predicate exists \cite[Theorem 2.4]{CP98}. The explicit axiomatization provided in their theorem yields immediately that the formula $\chi(\bar z,\bar z')=\exists x (x=x)$ with boundary the empty set works. 
\end{enumerate}
\end{example}

\begin{definition}\label{D:prim_qf}
A \emph{basic} formula in $n$ variables in the language $\LL_P$ is a formula  of the form 
\[ 
\theta(x_1,\ldots,x_n) \land \bigwedge\limits_{i\in I} P^{\epsilon_i}(t_i(x_1,\ldots,x_n)),
\] 
for a finite number of terms $t_i(\bar x)$, the formula $\theta(\bar x)$ is a  finite conjunction of atomic $\LL$-formulae (that is, a finite system of congruences) and each $\epsilon_i$ belongs to $\{-1,1\}$, with the convention that $P^1=P$ and $P^{-1}=\neg P$. A basic $\LL_P$-formula without any built-in congruences (that is, with no $\theta(\bar x)$ as above, given only by a conjunction of terms in the predicate and negations of such) is called \emph{primary}. 

Analogously, we define \emph{basic} formulae in the language $\LL_P^+$ as formulae of the form 
\[ 
\theta(\bar x) \land  \bigwedge\limits_{\mathclap{\substack{ i\in I \\ k\in J}}}  P_k^{\epsilon_{i, k}}(t_{i,k}(\bar x)),
\] 
for some finite subsets $I$ and $J$ of $\N$.  By an abuse of notation,  we say that a basic $\LL_P$-formula (resp. a basic $\LL_P^+$-formula) is over the $\LL_P^+$-substructure $A$ if it is an instance of a basic $\LL_P$-formula (resp. a basic $\LL_P^+$-formula) with parameters in $A$.
\end{definition}

\begin{lemma}\textup{(}\cf \cite[Claim 2.9]{KS17}\textup{)}~\label{L:Ziegler_KS}
Assume that the substructure $A=\strp{A}$ is pure in $\UU$ and consider a basic $\LL_P^+$-formula $\varphi(x, \bar a)$ in $\tp(b/A)$, where $b$ is a singleton. Then there are a primary $\LL_P$-formula $\psi(x, \bar a_1)$ over $A$, a non-constant $\LL$-term $t(x)=t(x, \bar a_1)$ over $A$ and an element $e$ with $t(e)=b$ such that
\begin{itemize}
\item the formula $\psi(x, \bar a_1)$ belongs to $\tp(e/A)$; 
\item For every $\LL_P^+$-isomorphism $F:A\to A'$ between the substructures $A$ and $A'$ of $\UU$, we have that $\psi(\UU, F(\bar a_1))\subset \varphi(t(\UU,F(\bar a_1)),F(\bar a))$. In particular, the subset $\psi(\UU, \bar a_1)$ is contained in $\varphi(t(\UU,\bar a_1), \bar a)$. 
\end{itemize}
\end{lemma}
\begin{proof}
By Remark \ref{R:trivial_obs}, we may assume that there is a positive integer $m\ge 1$ such that every term $t_{i,k}(x,\bar a)$ occurring in $\varphi$ in which $x$ occurs non-trivially is of the form $mx+a_{i,k}$. Given a term $mx+a$ with $a$ in $A$, observe that $\neg P_k(mx+a)$ is equivalent to 
\[ 
\bigvee_{j=1}^{k-1} mx+a\equiv_k j \vee  \exists y (ky = mx+a  \land \neg  P(y) \big). 
\] 
We will modify the formula $\varphi(x, \bar a)$ to obtain a new basic $\LL_P^+$-formula $\varphi_0(x,\bar a)$ as follows: if $\neg P_k(m x + a_{i,k})$ occurs in $\varphi(x,\bar a)$ and $m b+a_{i,k}\not \equiv_k 0$, then remove the formula $\neg P_k(m x + a_{i,k})$ and add $m x+a_{i,k} \equiv_k j$ to $\theta$ for a suitable $j=1+\stackrel{j}{\ldots}+1$ in $A$. Otherwise, leave the terms and subformulae unmodified. In particular, the modified formula $\theta$ so obtained is still a finite conjunction of atomic $\LL$-formulae. Note that $\varphi_0(x,\bar a)$ belongs to $\tp(b/A)$ and implies $\varphi(x,\bar a)$.  Thus, by the previous discussion,  setting $\varphi=\varphi_0$,  we may assume that $\theta(x, \bar a)$ encodes that every term $m x + a_{i,k}$ is divisible by $k$, for every occurrence (positive or negative) of the form $P_k(m x + a_{i,k})$ in $\varphi(x,\bar a)$. 

By the chinese remainder theorem, we may assume that $\theta(x, \bar a)$ is a single congruence $\ell x + a\equiv_n 0$ with $1\le \ell$ in $\N$, so $\ell b+ a= n \xi$ for some $\xi$ in $\UU$. Set $d$ the greatest common divisor of $\ell$ and $n$, so the element $a$ is divisible by $d$. Since $A$ is pure, there exists some $a'$ in $A$ with $d a'=a$, so setting $\ell'=\ell/d$ and $n'=n/d$ we get $\ell' b + a'= n' \xi$. By Bezout's identity, there exists some integer $k$ with $k \ell'\equiv_{n'} 1$. Rewriting, we conclude that $b+a''=n' b_1$ for some $b_1$  in $\UU$ and $a''=k a'$ in $A$, so the term $t_1(x, a'')= n' x-a''$ satisfies that $t_1(b_1, a'')=b$.  By construction, we have that $\ell t_1(x, a'') + a\equiv_n 0$, so $t_1(\UU, a'')$ lies in $\theta(\UU, \bar a)$. Set now 
\begin{multline*}
    \varphi_1= \bigwedge\limits_{\mathclap{\substack{ i\in I\\ k\in K}}}  P_k^{\epsilon_{i,k}}(mt_1(x) +a_{i, k})= \bigwedge\limits_{\mathclap{\substack{ i\in I \\ k\in K}}}  P_k^{\epsilon_{i,k}}(mn'x +(a_{i,k}-ma''))  \\ =  \bigwedge\limits_{\mathclap{\substack{ i\in I \\ k\in K}}}  P_k^{\epsilon_{i,k}}(mn'x +a'_{i,k})  \in \tp(b_1/A),
\end{multline*}
with $a'_{i,k}=a_{i,k}-ma''$ (so $\varphi(x, \bar a)=\varphi_0(x,\bar a)=\theta(x) \land \varphi_1(x)$). If no $P_k$ with $k\ge 2$ occurs in $\varphi_1$ (positively or negatively), then $\varphi_1$ is already a primary $\LL_P$-formula. Set thus $e=b_1$, $t=t_1$ and $\psi=\varphi_1$, so the conclusion of the theorem clearly holds. 

Otherwise, let $N\ge 2$ be the least common multiple of all the indices $k$'s such that $P_k$ (or its negation) occurs in $\varphi_1$. The element $b_1$ is congruent to some natural number $m_0=1+\stackrel{m_0}{\ldots}+1$ in $A$ modulo $N$, so set 
\[ 
e=\frac{b_1-m_0}{N},\ \text{ and hence } \ s(e) = b_1 \ \text{ for } s(x)=Nx+m_0.
\]  The term $t(x, a'')=t_1(s(x), a'')$ satisfies that $t(e, a'')=b$ and $t(\UU, a'')$ lies in $\theta(\UU, \bar a)$ by construction. In order to prove the statement, it suffices therefore to find a primary $\LL_P$-formula $\psi(x, \bar a_1)$ in $\tp(e/A)$ such that $\psi(\UU, F(\bar a_1))$ is contained in $\varphi_1(t(\UU, F(a'')), F(\bar a))$ for every $\LL_P^+$-isomorphism $F:A\to A'$, since the conjunction $\theta(x, \bar a)\land \varphi_1(x, \bar a)$ implies our original formula $\varphi(x, \bar a)$. 

Consider a term $m n' x+ a'_{i,k}$ occurring in the formula $\varphi_1$. Note that the element 
\[ 
mb+a_{i,k} = m n' b_1+ a'_{i,k}= m n' N e + a'_{i,k}+n m' m_0 
\] 
is divisible by $k$ by the first paragraph of the proof.  Since $k$ divides $N$, we have that  $a'_{i,k} + m n' m_0$ is divisible by $k$ (in $\UU$ and hence in $A$, as $A$ is pure). Set now $d_{i,k}=\frac{a'_{i,k}+m n' m_0}{k}$. In particular, 
\[
m b+ a_{i,k}=k(m n'\frac{N}{k} e+ d_{i, k}) \text{ satisfies $P_k$} \ \Leftrightarrow \ m n'\frac{N}{k} e+ d_{i, k} \text{ satisfies $P$}.
\] 
Extend the tuple $\bar a$ to $\bar a_1$ containing all $d_{i, k}$'s and consider the primary $\LL_P$-formula 
\[ 
\psi(x, \bar a_1)= \bigwedge  \limits_{\mathclap{\substack{ i\in I \\ k\in K}}}  P^{\epsilon_{i, k}}\Big(m n' \frac{N}{k} x + d_{i, k}\Big) \in \tp(e/A).
\]  
Assume now we have an $\LL_P^+$-isomorphism $F:A\to A'$ and choose some element $e'$ realizing $\psi(x, F(\bar a_1))$. By construction, we conclude that 
\begin{align*}
    m n'\frac{N}{k} e'+ F(d_{i, k}) \text{ satisfies $P$}\ &\Leftrightarrow \  k \Big(m n'\frac{N}{k} e'+ F(d_{i, k})\Big) \text{ satisfies $P_k$}  \\
   &\Leftrightarrow \ m  t(e',F(a'')) +F(a_{i,k}) \text{ satisfies $P_k$}, 
\end{align*}
as desired. 
\end{proof}
We finish this section with an easy observation on the definable set given by a primary formula, which will be useful in Section \ref{S:QE} to show the supersimplicity of the theory of the pair. 
\begin{definition}\label{D:good_pos}
An instance of a primary $\LL_P$-formula of the form
\[ 
\bigwedge_{i=1}^r P(m_i x+ c_i) \land  \bigwedge_{j=1}^{r'} \neg P(m'_j x+ c'_j),
\]
is \emph{in good position} if $(m_i, c_i)\ne (m'_j, c'_j)$ for $i\ne j$.     
\end{definition} 
Clearly, every consistent primary $\LL_P$-formula is in good position.

\begin{lemma}\label{L:primary_infty}
Let $\M$ be an elementary extension of the $\LL$-structure $\mathcal Z$. Consider an expansion of $\mathcal M$ by a new predicate $P$ which is interpreted as an additively random subset of $M$. Given $n\ge 1$ and primary $\LL_P$-formulae $\psi_1,\ldots, \psi_n$, each with parameters in a subset $A_i$, which are \emph{compatible}, that is, no term occurring positively in some $\psi_i$ occurs negatively in some $\psi_j$ with $i\ne j$, either $\bigwedge_i  \psi_i$ has infinitely many realizations or all the realizations of some $\psi_i$ are contained in the pure closure $\strp{A_i}$ of $A_i$.
\end{lemma}

\begin{proof}
Note that the conjunction $\psi=\bigwedge_i \psi_i$ is again a primary $\LL_P$-formula. Since the formulae $\psi_1,\ldots,\psi_n$ are compatible, the formula $\psi$ is in good position if and only if each of the formulae $\psi_i$ is. 

Now, if  $\psi$ is not in good position, it is inconsistent and has no realizations. Assume therefore, that $\psi$ is in good position and write it as
 \[ 
\bigwedge_{i=1}^r P(m_i x+ c_i) \land  \bigwedge_{j=1}^{r'} \neg P(m'_j x+ c'_j).
\]  
The pairs $(m_i, c_i)$ and $(m'_j, c'_j)$ are pairwise distinct. By Definition \ref{D:add_gen}, if $\chi(\bar z, \bar z')$ is the associated formula to $\psi$, we have that either $\psi$ has infinitely many realizations or  $\chi(\bar c,\bar c')$ does not hold. In the latter case, we deduce from Definition \ref{D:add_gen} (P2) that there is some $k$ in the boundary $\bound(\chi)$  such that for every integer $s$ in $\{0,\ldots, k-1\}$, some element $m_{i_0} s+c_{i_0}$ with $1\le i_0\le r$  must be divisible by $k$. Note that the term $m_{i_0} x+c_{i_0}$ belongs to some formula $\psi_i$, so $c_{i_0}$ lies in $A_i$. Choose any realization $e$ of $\psi_i$ and let $s$ be an integer with $e\equiv_k s$. The element $m_{i_0} e +c_{i_0}$ is divisible by $k$, so it must be an integer by Definition \ref{D:add_gen} (P3). Hence, we conclude that $e$ belongs to $\strp{A_i}$, as desired. 
\end{proof}

\section{Quantifier elimination and supersimplicity}\label{S:QE}
Given an elementary extension $\mathcal M$ of the $\LL$-structure $\mathcal Z$ and an additively random subset $D$ of $M$,  we will first show that the $\LL_P$-theory of the pair $(\mathcal M, D)$ has quantifier elimination up to basic $\LL_P^+$-formulae, that is, every $\LL_P$-formula $\varphi(x_1,\ldots, x_n)$ is a boolean combination of basic $\LL_P^+$-formulae in the same variables. 

\begin{prop}\label{P:BF_random}
Consider an expansion of the $\LL$-structure $\mathcal M$ by a new predicate $P$ which is interpreted as an additively random subset of $M$. The $\LL_P$-theory of $(\mathcal M,P^{\mathcal M})$ has quantifier elimination up to basic $\LL_P^+$-formulae.

In particular, the algebraic closure of a subset $A$ in the sense of $\LL_P$ coincides with the pure closure in $\mathcal M$ of the group $\sbgp{A,1}$. 
\end{prop}
\begin{proof}
The last statement follows directly from the quantifier elimination result by Lemma \ref{L:Ziegler_KS}. Indeed: if the element $b$ were algebraic over the subset $A$ in the $\LL_P$-structure $(\mathcal M, D)$, where $D=P^{\mathcal M}$, then there is by quantifier elimination a basic $\LL_P^+$-formula $\varphi(x, \bar a)$ witnessing that $b$ is algebraic over $A$. Lemma \ref{L:Ziegler_KS} yields a term $t(x, \bar a_1)$, an element $e$ with $b=t(e, \bar a_1)$ and a primary $\LL_P$-formula $\psi(x, \bar a_1)$ in $\tp(e/\strp A)$ for some $\bar a_1$ in $\strp A$ such that $\psi(\UU, \bar a_1)$ is contained in $\varphi(t(\UU), \bar a)$. It suffices to show that $e$ (and thus $b$) belongs to $\strp A$, since the term is not constant. Now, the formula $\psi(x, \bar a_1)$ is consistent and thus in good position, yet it is algebraic, since $t$ is not constant. By Lemma \ref{L:primary_infty}, we deduce that  $e$ belongs to $\strp A$, as desired. 

In order to show quantifier elimination, a standard application of the Separation Lemma \cite[Lemma 3.1.1]{TZ12} yields that it suffices to show that working inside sufficiently saturated models $\UU$ and $\UU'$ of $\mathrm{Th_{\LL_P}((\mathcal M, D)}$), the collection of maps $F:A\to A'$, where $A\subset \UU$ and $A'\subset \UU'$ are countable $\LL_P^+$-substructures such that $F$ is an $\LL_P^+$-isomorphism, is a non-empty Back-\&-Forth system. It is clearly non-empty, since there is a unique $\LL_P^+$-structure with universe the group $\Z$, as we have added a constant to the language for the element $1$. 

Let us now prove that this system satisfies the Forth-condition, for the Back-condition is analogous. Consider two countable $\LL_P^+$-substructures $A$ and $A'$ as well as an $\LL_P^+$-isomorphism  $F:A\to A'$. We may assume that $A$ and $A'$ are both pure in $\UU$ and $\UU'$ respectively, for the map $F$ extends uniquely to an $\LL_P^+$-isomorphism between the pure closures, since an element $e$ of $\UU$ with $m e=a$ in $A$ belongs to $P_k$ if and only if $a$ belongs to $P_{km}$.

Given an element $b$ in $\UU$, consider the partial type $\Sigma(x)$ over $A$ of all basic $\LL_P^+$-formulae  in $\tp(b/A)$. We need only show that there exists an element $b'$ satisfying the image of $\Sigma$ under $F$. If so, then there is a unique extension of $F$ to an $\LL_P^+$-isomorphism  $\bar{F}: \str{A, b}_{\LL_P^+}\to \str{A', b'}_{\LL_P^+}$  which maps $b$ to $b'$.

Without loss of generality, we may assume that $b$ is not in $A$. 
By compactness (for $\Sigma$ is closed under finite conjunctions), we need only consider a single $\LL_P^+$-formula $\varphi(x, \bar a)$ in $\Sigma$. Now, by Lemma \ref{L:Ziegler_KS}, write $b=t(e, \bar a_1)$ with $\bar a_1$ in $A$ and find a primary $\LL_P$-formula 
\[ 
\psi(x, \bar a_1)= \bigwedge_{i=1}^r P(m_i x+ c_i) \land  \bigwedge_{j=1}^{r'}\neg  P(m'_j x+ c'_j)
\]  
in $\tp(e/A)$ such that $\psi(\UU, F(\bar a_1))$ is contained in $\varphi(t(\UU, F(\bar a_1), F(\bar a))$. Note that $e$ does not belong to $A$ either, so the consistent primary formula $\psi(x, \bar a_1)$ is in good position and must have infinitely many realizations in $\UU$, by Lemma \ref{L:primary_infty}. By Remark \ref{R:add_gen}, we have that $\LL$-formula $\chi(\bar z, \bar z')$ associated to $\psi(x, \bar y)$ holds for $(\bar c, \bar c')$, and hence $\chi(F(\bar c), F(\bar c'))$ also holds in $\UU'$, by quantifier elimination of the $\LL$-theory of $\mathcal Z$. We conclude from property (P1) that the formula $\psi(x, F(\bar a))$ has infinitely many realizations in $\UU'$. Choose thus by compactness a realization $e'$ not lying in $A'$. The element $b'=t(e', F(\bar a_1))$ realizes $\varphi(x, F(\bar a))$, as desired. 
\end{proof}

We now have all the ingredients to show that the $\LL_P$-theory $\mathrm{Th}_{\LL_P}(\mathcal M, D)$ is supersimple of rank $1$. We first recall Shelah's definitions of dividing \cite{sS90}, working inside a sufficiently saturated model $\mathbb U$ of $\mathrm{Th}_{\LL_P}(\mathcal M, D)$.

\begin{definition}\label{D:forking_dividing}

A formula $\varphi(x, \bar a)$ \emph{divides} over the (small) subset $B$ of $\UU$ if there exists a $B$-indiscernible sequence $(\bar a_n)_{n\in \N}$ with $\bar a_0=\bar a$ such that the set $\{\varphi(x, \bar a_n)\}_{n\in \N}$ is inconsistent.

 The theory $\mathrm{Th}_{\LL_P}(\mathcal M, D)$ is \emph{supersimple of rank $1$} if every formula $\varphi(x, \bar a)$, with $x$ a single variable, whenever $\varphi(x, \bar a)$ divides over $B$, then $\varphi(x, \bar a)$ is algebraic (that is, the set $\varphi(\UU, \bar a)$ is finite). 
\end{definition}

\begin{theorem}\label{T:SU1}
Consider the $\LL_P$-expansion of an elementary extension $\mathcal M$ of the $\LL$-structure $\mathcal Z$ by interpreting $P$ as an additively random subset of $M$. The $\LL_P$-theory of $(\mathcal M,P^{\mathcal M})$ is supersimple of rank $1$. 
\end{theorem}

\begin{proof}
In order to show that every formula $\varphi(x, \bar a)$ which divides over the subset $B$ must be algebraic, we may assume by quantifier elimination that $\varphi(x, \bar a)$ is a consistent basic $\LL_P$-formula over $A=\strp {B, \bar a}$. We need only show that every realization $c$ of $\varphi(x, \bar a)$ belongs to $A$, by Proposition \ref{P:BF_random}.

By Lemma \ref{L:Ziegler_KS}, we find a non-constant term $t(x, \bar a_1)$ with $c=t(e, \bar a_1)$ for some $e$ in $\UU$ and $\bar a_1$ in $A$, as well as a primary $\LL_P$-formula $\psi(x, \bar a_1)$ in $\tp(e/A)$  such that  $\psi(\UU, \bar a_1)$ is contained in $\varphi(t(\UU), \bar a)$. Note that the $B$-indiscernible sequence for $\varphi(x, \bar a)$ induces a $B$-indiscernible sequence for $\psi(x, \bar a_1)$ witnessing that $\psi(x, \bar a_1)$ divides over $B$. It suffices to show that $\psi(x, \bar a_1)$ is algebraic, so $e$ (and thus $c$) belongs to $A$, as desired. Rewrite the formula as 
\[ 
\psi(x, \bar a, \bar a') = \bigwedge_{i=1}^r P(m_i x+ a_i) \land  \bigwedge_{j=1}^{r'} \neg P(m'_j x+ a'_j)  
\] 
and notice that it is in good position, since it is consistent. By assumption, the formula $\psi(x, \bar a, \bar a')$ divides over $B$, witnessed by the sequence $(\bar a_n, \bar a'_n)_{n\in \N}$. By compactness, there exists some $N\ge 2$ in $\N$ such that each of the primary $\LL_P$-formula $\psi(x, \bar a_k, \bar a'_k)$ with $k\le N$ is in good position (by indiscernibility), yet the conjunction $\bigwedge_{k\le N} \psi(x, \bar a_k, \bar a'_k)$ is inconsistent. It suffices to show that the formulae $(\psi(x, \bar a_k, \bar a'_k))_{k\le N}$ are compatible, by Lemma \ref{L:primary_infty}: indeed, all of the realizations of some $\psi(x, \bar a_i, \bar a'_i)$ must lie then in $\strp{B, \bar a_i, \bar a'_i}$, so  the formula $\psi(x, \bar a, \bar a')$ is algebraic over $A$ by indiscernibility, as desired.

Assume for a contradiction that the formulae $(\psi(x, \bar a_k, \bar a'_k))_{k\le N}$ are not compatible, and find, without loss of generality, two indices $1\le s<t\le N$ and some term which occurs positively in $\psi(x, \bar a_s, \bar a'_s)$ but negatively in $\psi(x, \bar a_t, \bar a'_t)$. After choosing suitable coordinates, write this term as $m_i x+ a_{s,i}=m'_j x+ a'_{t,j}$. 

By indiscernibility of the sequence, we have that 
\[ 
m'_j x+ a'_{t,j} = m_i x+ a_{s,i} = m'_j x+ a'_{t+1, j},
\] 
so 
\[  
m_i x+ a_{s,i}= m'_j x+ a'_{t,j}= m'_j x+ a'_{s,j},
\] contradicting that the formula $\psi(x, \bar a_i, \bar a'_i)$ was consistent, as desired.  
\end{proof}

We finish this note with the following question, for which we do not have a particular answer in mind.

\begin{question}
Consider the expansion $(\Z,0,+,-,\mathrm{Pr},\mathrm{SF})$ of the abelian group $\Z$, where $\mathrm{Pr}$ denotes all prime integers and $\mathrm{SF}$ denotes all square-free integers (see Example \ref{E:add-random_ex}). Does this expansion have a supersimple theory of rank $1$? 

More precisely, is there a sensible number-theoretical conjecture, (possibly) stronger than Dickson's conjecture, which implies that for any \emph{suitable} (\cf Definition \ref{D:add_gen}) collection of terms 
\[ m_1 x+ z_1,\ldots, m_r x+ z_r, m'_1 x + z'_1,\ldots, m'_{r'}x+z'_{r'}, \text{ and }m''_1 x + z''_1,\ldots, m''_{r''}x+z''_{r''},\]
whenever the associated terms are pairwise distinct, the formula 
 \[ 
    \bigwedge_{i=1}^r \mathrm{Pr}(m_i x+b_i) \land  \bigwedge_{j=1}^{r'}  (\mathrm{SF}(m'_j x+b'_j) \land \neg\mathrm{Pr}(m'_j x+b'_j)) \land    \bigwedge_{k=1}^{r''} \neg \mathrm{SF}(m''_k x+b''_k)
    \] 
has infinitely many realizations, unless a first-order condition expressible in the language $\LL$ of the abelian group $\Z$ with congruences holds?  If such a condition as above holds, we can just reproduce verbatim the proofs in this note to deduce the supersimplicity of the theory of $(\Z,\mathrm{Pr},\mathrm{SF})$. Actually, it would most likely suffice if there is an associated formula in the language $\LL\cup\{\mathrm{SF}\}$, by  considering the expansion of the supersimple theory of $(\Z,\mathrm{SF})$ by a predicate $\mathrm{Pr}$ for the prime integers. 
\end{question}

\end{document}